\documentclass[a4paper, 11pt]{amsart}
\pdfoutput=1
\usepackage{amssymb,centernot}
\usepackage[colorlinks=true,linkcolor=blue,citecolor=blue,urlcolor=cyan, pdfpagelabels=false]{hyperref}
\usepackage{tikz}
\usepackage{perpage} 
\MakePerPage{footnote}
\usetikzlibrary{decorations.shapes}
\usetikzlibrary{matrix}

\begin{document}
\numberwithin{equation}{section}
\newtheorem{theorem}{Theorem}
\newtheorem{algo}{Algorithm}
\newtheorem{lem}{Lemma} 
\newtheorem{de}{Definition} 
\newtheorem{ex}{Example}
\newtheorem{pr}{Proposition} 
\newtheorem{claim}{Claim} 
\newtheorem*{re}{Remark}
\newtheorem*{asi}{Aside}
\newtheorem{co}{Corollary}
\newtheorem{conv}{Convention}
\newcommand{\di}{\hspace{1.5pt} \big|\hspace{1.5pt}}
\newcommand{\idi}{\hspace{.5pt}|\hspace{.5pt}}
\newcommand{\hs}{\hspace{1.3pt}}
\newcommand{\thmf}{Theorem~1.15$'$}
\newcommand{\ndi}{\centernot{\big|}}
\newcommand{\nidi}{\hspace{.5pt}\centernot{|}\hspace{.5pt}}
\newcommand{\lpp}{\mbox{$\hspace{0.12em}\shortmid\hspace{-0.62em}\alpha$}}
\newcommand{\btt}{\mbox{$\raisebox{-0.59ex}
  {${{l}}$}\hspace{-0.215em}\beta\hspace{-0.88em}\raisebox{-0.98ex}{\scalebox{2}
  {$\color{white}.$}}\hspace{-0.416em}\raisebox{+0.88ex}
  {$\color{white}.$}\hspace{0.46em}$}{}}
  \newcommand{\un}{\hs\underline{\hspace{5pt}}\hs}
\newcommand{\lp}{\widehat{\lpp}}
\newcommand{\bt}{\hspace*{2pt}\widehat{\hspace*{-2pt}\btt}}
\newcommand{\PQ}{\bb{P}^1(\bb{Q})}
\newcommand{\pmn}{\cl{P}_{M,N}}
\newcommand{\he}{holomorphic eta quotient\hspace*{2.5pt}}
\newcommand{\hes}{holomorphic eta quotients\hspace*{2.5pt}}
\newcommand{\defG}{Let $G\subset\GG$ be a subgroup that is conjugate to a finite index subgroup of $\G$. } 
\newcommand{\defg}{Let $G\subset\GG$ be a subgroup that is conjugate to a finite index subgroup of $\G$\hs\hs} 
\renewcommand{\phi}{\varphi}
\newcommand{\Z}{\bb{Z}}
\newcommand{\ZD}{\Z^{\D}}
\newcommand{\N}{\bb{N}}
\newcommand{\Q}{\bb{Q}}
\newcommand{\A}{\widehat{A}}
\newcommand{\pii}{{{\pi}}}
\newcommand{\R}{\bb{R}}
\newcommand{\C}{\bb{C}}
\newcommand{\I}{\hs\cl{I}_{n,N}}
\newcommand{\St}{\operatorname{Stab}}
\newcommand{\D}{\cl{D}_N}
\newcommand{\rh}{{{\boldsymbol\rho}}}
\newcommand{\bh}{{\cl{M}}} 
\newcommand{\lv}{\hyperlink{level}{{\text{level}}}\hspace*{2.5pt}}
\newcommand{\fct}{\hyperlink{factor}{{\text{factor}}}\hspace*{2.5pt}}
\newcommand{\q}{\hyperlink{q}{{\mathbin{q}}}}
\newcommand{\rd}{\hyperlink{redu}{{{\text{reducible}}}}\hspace*{2.5pt}}
\newcommand{\ird}{\hyperlink{irredu}{{{\text{irreducible}}}}\hspace*{2.5pt}}
\newcommand{\str}{\hyperlink{strong}{{{\text{strongly reducible}}}}\hspace*{2.5pt}}
\newcommand{\rdn}{\hyperlink{redon}{{{\text{reducible on}}}}\hspace*{2.5pt}}
\newcommand{\atl}{\hyperlink{atinv}{{\text{Atkin-Lehner involution}}}\hspace*{3.5pt}}
\newcommand{\T}{\mathrm{T}}
\newcommand{\nm}{{N,M}}
\newcommand{\mn}{{M,N}}
\renewcommand{\H}{\fr{H}}
\newcommand{\W}{\text{\calligra W}_n}
\newcommand{\GG}{\cl{G}}
\newcommand{\g}{\fr{g}}
\newcommand{\Gm}{\Gamma}
\newcommand{\Gmtl}{\widetilde{\Gamma}_\ell}
\newcommand{\gm}{\gamma}
\newcommand{\go}{\gamma_1}
\newcommand{\gmt}{\widetilde{\gamma}}
\newcommand{\gmdt}{\widetilde{\gamma}'}
\newcommand{\gmot}{\widetilde{\gamma}_1}
\newcommand{\gmdot}{{\widetilde{\gamma}}'_1}
\newcommand{\s}{\Large\text{{\calligra r}}\hspace{1.5pt}}
\newcommand{\ms}{m_{{{S}}}}
\newcommand{\nisim}{\centernot{\sim}}
\newcommand{\level}{\hyperlink{level}{{\text{level}}}}
\newcommand{\Redcon}{the \hyperlink{red}{\text{Reducibility~Conjecture}}}
\newcommand{\Conred}{Conjecture~$1'$}
\newcommand{\Conredd}{Conjecture~$1''$}
\newcommand{\Conreddd}{Conjecture~$1'''$}
\newcommand{\Conired}{Conjecture~$2'$}
\newtheorem*{pro}{\textnormal{\textit{Proof of the proposition}}}
\newtheorem*{cau}{Caution}
\newtheorem*{conjec}{Conjecture}
\newtheorem{thrmm}{Theorem}[section]
\newtheorem{no}{Notation}
\renewcommand{\thefootnote}{\fnsymbol{footnote}}
\newtheorem{oq}{Open question}
\newtheorem{conj}{Conjecture}
\newtheorem{hy}{Hypothesis} 
\newtheorem{expl}{Example}
\newcommand\ileg[2]{\bigl(\frac{#1}{#2}\bigr)}
\newcommand\leg[2]{\Bigl(\frac{#1}{#2}\Bigr)}
\newcommand{\e}{\eta}
\newcommand{\sgn}{\operatorname{sgn}}
\newcommand{\bb}{\mathbb}
\newtheorem*{conred}{Conjecture~\ref{con1}$\mathbf{'}$}
\newtheorem*{conredd}{Conjecture~\ref{con1}$\mathbf{''}$}
\newtheorem*{conreddd}{Conjecture~\ref{con1}$\mathbf{'''}$}
\newtheorem*{conired}{Conjecture~\ref{19.1Aug}$\mathbf{'}$}
\newtheorem*{procl}{\textnormal{\textit{Proof}}}
\newtheorem*{thmbb}{Theorem~\ref{b1}$\mathbf{'}$}
\newcommand{\thmb}{Theorem~1$'$}
\newtheorem*{coo}{Corollary~\ref{17Aug}$\mathbf{'}$}
\newcommand{\cooo}{Corollary~\ref{17Aug}$'$}
\newtheorem*{cotw}{Corollary~\ref{17.1Aug}$\mathbf{'}$}
\newcommand{\cotww}{Corollary~\ref{17.1Aug}$'$}
\newtheorem*{cotb}{Corollary~\ref{b11}$\mathbf{'}$}
\newcommand{\cotbb}{Corollary~\ref{17.2Aug}$'$}
\newtheorem*{cne}{Corollary~\ref{15.5Aug}$\mathbf{'}$}
\newcommand{\cnew}{Corollary~\ref{15.5Aug}$'$\hspace{3.5pt}}
\newcommand{\fr}{\mathfrak}
\newcommand{\cl}{\mathcal}
\newcommand{\rad}{\mathrm{rad}}
\newcommand{\ord}{\operatorname{ord}}
\newcommand{\m}{\setminus}
\newcommand{\G}{\Gamma_1}
\newcommand{\GN}{\Gamma_0(N)}
\newcommand{\X}{\widetilde{X}}
\renewcommand{\P}{{\textup{p}}} 
\newcommand{\al}{{\hs\operatorname{al}}}
\newcommand{\p}{p_\text{\tiny (\textit{N})}}
\newcommand{\pN}{p_\text{\tiny\textit{N}}}
\newcommand{\U}{u_\textit{\tiny N}}
\newcommand{\Upr}{u_{\textit{\tiny N}^\prime}}
\newcommand{\Up}{u_{\textit{\tiny p}^\textit{\tiny e}}}
\newcommand{\Un}{u_{\textit{\tiny p}_\textit{\tiny 1}^{\textit{\tiny e}_\textit{\tiny 1}}}}
\newcommand{\Um}{u_{\textit{\tiny p}_\textit{\tiny m}^{\textit{\tiny e}_\textit{\tiny m}}}}
\newcommand{\Ut}{u_{\text{\tiny 2}^\textit{\tiny a}}}
\newcommand{\At}{A_{\text{\tiny 2}^\textit{\tiny a}}}
\newcommand{\Uh}{u_{\text{\tiny 3}^\textit{\tiny b}}}
\newcommand{\Ah}{A_{\text{\tiny 3}^\textit{\tiny b}}}
\newcommand{\Uprl}{u_{\textit{\tiny N}_1}}
\newcommand{\Uprlm}{u_{\textit{\tiny N}_i}}
\newcommand{\UM}{u_\textit{\tiny M}}
\newcommand{\UMp}{u_{\textit{\tiny M}_1}}
\newcommand{\w}{\omega_\textit{\tiny N}}
\newcommand{\wm}{\omega_\textit{\tiny M}}
\newcommand{\wa}{\omega_{\text{\tiny N}_\textit{\tiny a}}}
\newcommand{\wma}{\omega_{\text{\tiny M}_\textit{\tiny a}}}
\renewcommand{\P}{{\textup{p}}} 
\tikzset{decorate sep/.style 2 args=
{decorate,decoration={shape backgrounds,shape=circle,shape size=#1,shape sep=#2}}}

\title[\tiny{Finiteness of irreducible holomorphic eta quotients}] 
{Finiteness of irreducible holomorphic eta quotients of a given level}

\author{Soumya Bhattacharya}
\address 
{CIRM : FBK\\
via Sommarive 14\\
I-38123 Trento}

\email{soumya.bhattacharya@gmail.com}
\subjclass[2010]{Primary 11F20, 11F37, 11F11; Secondary 
11G16, 11F12}

 \begin{abstract}
   We show that for any positive integer $N$,
   there are only finitely many
holomorphic eta quotients of level $N$,
  none of which is a product of two 
  holomorphic eta quotients other than 1 and itself.
  This result is an analog of Zagier's conjecture/ 
  Mersmann's theorem 
  which states that: Of any given weight,
  there are
  only finitely many irreducible holomorphic eta quotients, none of which is an integral rescaling of 
  another eta quotient. We 
  construct 
  such eta quotients 
  for all cubefree levels. 
  In particular, 
  our construction 
  demonstrates the existence of irreducible holomorphic eta quotients of arbitrarily large
  weights.
 \end{abstract}

\maketitle

 \section{Introduction}
 \label{intro}
 The Dedekind eta function is defined by the infinite product:\hypertarget{queue}{}
 \begin{equation} \eta(z):=q^{\frac{1}{24}}\prod_{n=1}^\infty(1-q^n)\hspace{5pt}\text{for all $z\in\H$,}
\label{17.4Aug}\end{equation} 
where $q^r=q^r(z):=e^{2\pi irz}$ for all $r$ 
and 
$\H:=\{\tau\in\C\hs\idi\operatorname{Im}(\tau)>0\}$. 
Eta is a holomorphic function on $\H$ with no zeros.
This function has its significance in Number Theory. For example, 
$1/\e$ is the generating function for the ordinary partition function $p:\N\rightarrow\N$
(see \cite{ak})
and 
 the constant term in the Laurent
 expansion at 
 $1$ of 
the Epstein zeta function $\zeta_Q$ 
attached to a positive definite quadratic form $Q$ is related
 via the Kronecker limit formula to
the value of $\e$
at the root of the associated quadratic polynomial in $\H$ (see \cite{coh}).
The value of $\e$ at such a quadratic irrationality of discriminant $-D$
is also related via the Lerch/Chowla-Selberg formula to  the values of the Gamma function 
with arguments in $D^{-1}\N$ 
 \hspace{.5pt}(see 
  \cite{pw}). 
In fact, the eta function
comes up naturally in many other areas of Mathematics (see the Introduction in \cite{B-three} for a brief
overview of them). 

The function $\e$ is a modular form
of weight $1/2$ with a multiplier system 
on $\operatorname{SL}_2(\Z)$ (see \cite{b}).
An 
eta quotient $f$ is a finite product of the form 
\begin{equation}
 \prod\e_d^{X_d},
\label{13.04.2015}\end{equation}
where $d\in\N$, $\eta_d$ is the rescaling of $\eta$ by $d$, defined by
\begin{equation}
 \e_d(z):=\e(dz) \ \text{ for all $z\in\H$}
\end{equation}
and $X_d\in\Z$.
Eta quotients naturally inherit modularity 
from $\e$: The eta quotient $f$ in (\ref{13.04.2015}) transforms like a modular form of
weight $\frac12\sum_dX_d$ with a multiplier system on suitable congruence subgroups of $\operatorname{SL}_2(\Z)$: 
The largest among
these subgroups is 
\begin{equation}
 \Gm_0(N):=\Big{\{}\begin{pmatrix}a&b\\ c&d\end{pmatrix}\in
\operatorname{SL}_2(\Z)\hspace{3pt}\Big{|}\hspace{3pt} c\equiv0\hspace*{-.3cm}\pmod N\Big{\}},
\end{equation}
where 
\begin{equation}
 N:=\operatorname{lcm}\{d\in\N\hs\idi\hs X_d\neq0\}.
\end{equation}
We call $N$ the \emph{level} of $f$.
Since $\eta$ is non-zero on $\H$, 
the eta quotient $f$ 
is holomorphic if and only if $f$ does not have any pole at the cusps of 
$\Gamma_0(N)$.

An \emph{eta quotient 
on $\Gm_0(M)$} is
an eta quotient whose level divides $M$.
Let $f$, $g$ and $h$ be nonconstant \hes on $\Gm_0(M)$
such that 
$f=g\times h$. Then we say that $f$ is \emph{factorizable on} $\Gm_0(M)$. 
We call a holomorphic eta quotient $f$ of level $N$ \emph{quasi-irreducible} (resp. \emph{irreducible}),
if it is not factorizable on $\Gm_0(N)$ (resp. on $\Gm_0(M)$ for all multiples $M$ of $N$).

Irreducible holomorphic eta quotients were first considered by Zagier, who conjectured (see \cite{z}) that:
\emph{There are only finitely many primitive and irreducible holomorphic eta quotients of a given weight.}
An eta quotient $f$ is called 
\emph{primitive}  
if 
no eta quotient $h$ and 
no integer $\nu>1$
satisfy the equation $f=h_\nu$, where $h_\nu(z):=h(\nu z)$ for all $z\in\H$.
Zagier's conjecture was 
established by his student Mersmann in 
an 
excellent \emph{Diplomarbeit} 
\cite{c}. 
I gave a simplified proof of 
this theorem
in \cite{B-five}. 
Unfortunately, none of 
the existing proofs of Mersmann's finiteness theorem
yield
an explicit upper bound 
for the levels of 
primitive and irreducible holomorphic 
eta quotients of a given weight.
However, 
another approach would be to
look at the problem from the dual perspective,
where instead of considering holomorphic eta quotients of a given weight,
we consider holomorphic eta quotients of a given level. 
If one could obtain a nontrivial estimate for the least possible weight for a primitive and irreducible holomorphic eta quotient of level $N$,
that would immediately imply
an effective proof of Mersmann's finiteness theorem. 
For example, if $p$ is a prime and if $N=p$ or $p^2$, then the least possible weight of a primitive and irreducible holomorphic
eta quotient of level $N$ is $(p-1)/2$ (see Section~6.3 in~\cite{B-two}).
Though  the notions of irreducibility and quasi-irreducibility of holomorphic 
\mbox{eta quotients are conjecturally equivalent (see~\cite{B-three}), in practice}
irreducibility of a holomorphic eta quotient is much harder to determine than
its quasi-irreducibility. 
However, since every irreducible holomorphic eta quotient is quasi-irreducible,
the least possible weight of a primitive and irreducible holomorphic eta quotient
of level $N$ is 
bounded 
below by the least possible weight of a primitive and 
quasi-irreducible holomorphic eta quotient of level $N$. 
We denote the later by $k_{\min}(N)/2$. 
\begin{table}
 \caption{$N$ vs. $k_{\min}(N)$}
\begin{center}
  \begin{tabular}{ c | c } 
N&
$k_{\min}$ 
\\\hline
$2\cdot3$&1\\$2\cdot5$&2\\$2^2\cdot3$&1\\$2\cdot7$&3\\$3\cdot5$&2\\$2^4$&2\\$2\cdot3^2$&2\\$2^2\cdot5$&2\\$3\cdot7$&3\\$2\cdot11$&4\\$2^3\cdot3$&2\\$2\cdot13$&5\\$2^2\cdot7$&3\\$2\cdot3\cdot5$&2\\$2^5$&2\\$2\cdot17$&6\\$2^2\cdot3^2$&2\\$2\cdot19$&7\\
$3\cdot13$&5\\$2^3\cdot5$&2\\
$2\cdot3\cdot7$&2\\$2^2\cdot11$&3\\$3^2\cdot5$&2\\$2\cdot23$&8\\$2^4\cdot3$&2\\$2\cdot5^2$&2\\$3\cdot17$&6
\end{tabular}
\quad
  \begin{tabular}{  c | c  } 
N&
$k_{\min}$ 
\\\hline
$2^2\cdot13$&4\\$2\cdot3^3$&2\\$2^3\cdot7$&2\\$3\cdot19$&6\\$2^2\cdot3\cdot5$&2\\$3^2\cdot7$&2\\$2^6$&2\\$2\cdot3\cdot11$&3\\$2^2\cdot17$&4\\$2\cdot5\cdot7$&2\\$2^3\cdot3^2$&2\\$2\cdot37$&13\\$3\cdot5^2$&3\\$2\cdot3\cdot13$&3\\$2^4\cdot5$&2\\
$3^4$&3\\$2^2\cdot3\cdot7$&2\\$5\cdot17$&6\\$2^3\cdot11$&2\\$2\cdot3^2\cdot5$&2\\$2\cdot47$&16\\$2^5\cdot3$&2\\$2\cdot7^2$&3\\$3^2\cdot11$&4\\$2^2\cdot5^2$&2\\$2\cdot3\cdot17$&3\\$2^3\cdot13$&3
\end{tabular}
\quad
  \begin{tabular}{  c | c  } 
N&
$k_{\min}$ 
\\\hline
$3\cdot5\cdot7$&3\\$2^2\cdot3^3$&2\\$3\cdot37$&11\\$2^4\cdot7$&2\\$2^3\cdot3\cdot5$&2\\$2\cdot3^2\cdot7$&2\\$2^7$&3\\$2\cdot5\cdot13$&3\\$7\cdot19$&8\\$3^3\cdot5$&3\\$2^3\cdot17$&3\\$2\cdot3\cdot23$&3\\$2^2\cdot5\cdot7$&2\\$2^4\cdot3^2$&2\\$2^2\cdot37$&8\\$2\cdot3\cdot5^2$&2\\$2\cdot7\cdot11$&3\\$2\cdot3^4$&3\\
$2^3\cdot3\cdot7$&2\\$2\cdot5\cdot17$&4\\$2^2\cdot43$&9\\$2^4\cdot11$&2\\$2^2\cdot3^2\cdot5$&2\\$2\cdot7\cdot13$&3\\$2^3\cdot23$&3\\$3^3\cdot7$&3\\$2^6\cdot3$&2
\end{tabular}
\quad
  \begin{tabular}{  c | c  } 
N&
$k_{\min}$ 
\\\hline
$2^2\cdot7^2$&3\\$2^3\cdot5^2$&3\\$11\cdot19$&11\\$2^3\cdot3^3$&2\\$2^4\cdot3\cdot5$&2\\$3^5$&3\\$2^2\cdot3^2\cdot7$&2\\$2^8$&3\\$2^5\cdot3^2$&2\\$2^2\cdot3^4$&2\\$2^4\cdot3^3$&2\\$2^9$&3\\$2^6\cdot3^2$&2\\$5^4$&5\\$3^6$&3\\$2^8\cdot3$&2\\$2^{10}$&3\\$17\cdot97$&21\\$2^{11}$&3\\$3^7$&5\\$7^4$&7\\$5^5$&5\\$2^{12}$&3\\$3^8$&5\\$2^{13}$&3\\$2^{14}$&3\\$2^{15}$&4
\end{tabular}
\end{center}
\label{table}\end{table}
With a huge amount of numerical evidence similar to what we see below in 
Table~\ref{table}, we speculate that
 \begin{conj}
 For a positive integer $N$, if there exists a primitive and irreducible holomorphic eta quotient 
 of weight $k/2$ and level $N$, then 
\begin{equation}
 4k^2\geq\max_{\substack{p^n\|N\\p \text{ prime}}}np.
\label{chulchera}
\end{equation}
 \end{conj}
Mersmann's finiteness theorem is equivalent to say that for each $k\in\N$,
there exists an $M_k\in\N$ such that if there exists a primitive and irreducible
holomorphic eta quotient of weight less than or equal to $k/2$ and level $N$,
then $N$ divides $M_k$. From \cite{z}, we know that $M_1=12$. 
Also, from 
Corollary~A.1 
in the appendix 
of \cite{B-two}, we know\footnote{Unlike the general case, irreducibility of a holomorphic eta quotient of weight~1 is rather easy to determine, 
because 
a holomorphic eta quotient of weight~1 and level~$N$ is irreducible
if and only if it is not factorizable on $\Gm_0(\operatorname{lcm}(N,12))$ (see Lemma~1 in \cite{B-three}).
In particular, the irreducibility of the holomorphic eta quotients listed in 
Appendix~A in \cite{B-two}
could be easily verified.}
\begin{equation}
2^{8}\cdot3^3\cdot5^2\cdot7\cdot11\idi\hspace*{2pt}M_2.
\end{equation}
In particular,
the truth of the above conjecture would imply 
\begin{equation}
 M_2\idi\hspace*{1.5pt}2^{8}\cdot3^5\cdot5^3\cdot7^2\cdot11\cdot13. 
\end{equation}
Since there are only finitely many holomorphic eta quotients of a given weight and level (see~\ref{27.04.2015.3Z}),
knowing $M_k$ is equivalent to having a complete list 
of primitive and irreducible holomorphic eta quotients of weight~$k/2$,
up~to a bounded amount of computation.
About thirty years ago, Zagier gave such a list for
eta quotients of weight $1/2$ 
(see~\cite{z}), 
the exhaustiveness of 
which was also established 
by Mersmann (see~\cite{c, B-six}). Though Mersmann's finiteness theorem predicts the 
existence of similar lists of eta quotients for any given weight, till now
we do not even
have such an exhaustive list for 
primitive and irreducible holomorphic
eta quotients of weight $1$ (for an incomplete list, see Appendix~A in \cite{B-two}).
On the other hand, it is much
easier to list \emph{all} the irreducible holomorphic eta quotients of a given level! For example, let us look at
the following cases:
Since the only holomorphic eta quotients of level~1 are the powers of eta, $\e$ is the only 
irreducible holomorphic eta quotient of level~1.
In general, since 
the weight of
any holomorphic eta quotient
is at least $1/2$,
each \hypertarget{levp}{eta} quotient of weight $1/2$ is irreducible. 
In particular, for any $p\in\N$, the eta quotient $\e_p$ is irreducible.
Again, from Corollary~2 in \cite{B-three}, we know that for any prime $p$,
the holomorphic eta quotients $\e^p/\e_p$ and $\e^p_p/\e$ are irreducible 
(the irreducibility of the later also follows from Lemma~\ref{10.09.2015.1} below).
It is easy to show that any other holomorphic eta quotient of level $p$ 
except these three is factorizable on $\Gm_0(p)$. So,
the above three are the only irreducible \hes of a prime level $p$.
Here, we shall show that the
finiteness of irreducible holomorphic eta quotients 
of a given level
also holds in general. 
This in particular, implies 
that the maximum of the
weights of the irreducible holomorphic eta quotients of level $N$
is bounded above  with respect to $N$. 
Conversely, since the valence formula implies that there are only finitely many holomorphic
eta quotients of a given level and weight (see \ref{27.04.2015.3Z}),  
the finiteness of irreducible holomorphic eta quotients of a given level 
is also implied by such 
an upper bound. 
In particular,
the 
finiteness of quasi-irreducible holomorphic eta quotients of a given level
(see Theorem~\ref{b1}),
has an application
in \cite{B-three}, in showing that the levels of the factors of a holomorphic eta quotient $f$
are bounded above 
in terms of the level of $f$. 

Before ending this section, let us compare the situation with that of the modular 
forms with the trivial multiplier system. Note that the notions of irreducibility and factorizability
also makes sense if we replace ``holomorphic eta quotients'' with ``modular forms'' above. For example,
the 
modular form $\Delta:=\e^{24}$ of level~1 is not factorizable  into a product of modular forms with the trivial multiplier system on $\operatorname{SL}_2(\Z)$,
since $\Delta$ is a cusp form of the least possible weight on 
the full modular group.
However, $\Delta$ 
is factorizable on $\Gm_0(2):$
\begin{equation}
 \Delta=\e^8\e_2^8\times\frac{\e^{16}}{\e_2^8}.
\label{11.09.2015}\end{equation}
From (\ref{27.04}), it follows readily that 
$\e^p/\e_p$ is holomorphic for each prime $p$.
In particular, so is the 
rightmost eta quotient in 
(\ref{11.09.2015}). 
Also, 
from Newman's criteria (see \cite{new, newm} or \cite{rw}), it follows that the multiplier systems of 
both of the eta quotients on the right hand side of (\ref{11.09.2015}) are trivial.

For $k\in2\N$, we define 
the normalized Eisenstein series $E_k$ by 
\begin{equation}
E_k(z):=1-\frac{2k}{B_k}\sum_{n=1}^\infty\sigma_{k-1}(n)q^n,
\end{equation}
where the function $\sigma_{k-1}:\N\rightarrow\N$ is given by
\begin{equation}
\sigma_{k-1}(n):=\sum_{d\idi n}d^{k-1}
\end{equation}
and the $k$-th Bernoulli number $B_k$ is defined by 
\begin{equation}
\frac{t}{e^t-1}=\sum_{k=0}^{\infty}\frac{B_k}{k!}\cdot t^k.
\end{equation}
For each even integer $k>2$, $E_k$ is a modular form of weight $k$ on  $\operatorname{SL}_2(\Z)$ (see \cite{z}). 
Since there are no nonzero modular forms of 
odd weight or weight~2 
with the trivial multiplier system 
on  $\operatorname{SL}_2(\Z)$,
neither $E_4$ nor $E_6$ is factorizable on $\operatorname{SL}_2(\Z)$.
However, since $E_6(i)=0$ and $E_4(e^{2\pi i/3})=0$
and since the valence formula (\ref{3.2Aug}) for $\Gm_0(2)$ 
(resp. for $\Gm_0(4)$) implies that the modular form
\begin{equation}
  f_1:=\e_2(i)^{24}\hspace{1pt}\frac{\e^{16}}{\e_2^8}-\e(i)^{24}\hspace{1pt}\frac{\e_2^{16}}{\e^8}, \ \ 
  \text{resp. }   
  f_2:=\e_4(e^{\frac{2\pi i}{3}})^{8}\hspace{1pt}\frac{\e^{8}}{\e_2^4}-\e(e^{{\frac{2\pi i}{3}}})^{8}\hspace{1pt}\frac{\e_4^{8}}{\e_2^4} 
\end{equation}
of weight 4 on $\Gm_0(2)$ (resp. of weight 2 on $\Gm_0(4)$) only has a simple zero at $i$ in $\Gm_0(2)\backslash\H$ (resp. at $e^{2\pi i/3}$ in $\Gm_0(4)\backslash\H$), it follows that $f_1$ 
is a nontrivial factor of $E_6$ on $\Gm_0(2)$ (resp. $f_2$ is a nontrivial factor of $E_4$ on $\Gm_0(4)$).
It is easy to check that for all integers $N>1$, the stabilizers of $i$ 
and $e^{2\pi i/3}$ in $\Gm_0(N)$
are trivial.
The holomorphy of the eta quotients in the above linear combinations
follows trivially from (\ref{28.04}), 
once one notes the 
outermost columns of the
matrix in (\ref{r1}).
The triviality of the multiplier systems of these eta quotients 
follows again from Newman's criteria (see \cite{new, newm}~or~\cite{rw}).

Also, it follows 
from the valence formula (\ref{3.2Aug}) for $\operatorname{SL}_2(\Z)$ that 
every modular form 
with the trivial multiplier system
on $\operatorname{SL}_2(\Z)$ 
has a unique factorization of the form:
\begin{equation}
C_0\hspace*{1pt}E_4^aE_6^b\hspace*{-2pt}\prod_{t\in\C^*}\hspace*{-2pt}(E_4^3-tE_6^2)^{c_t},
\end{equation}
for some $C_0\in\C$ and some nonnegative integers $a,b,c_t$, where $c_t$ is zero
for all but finitely many $t$. 
In particular, any modular form with the trivial multiplier system and of weight greater than 12 on  $\operatorname{SL}_2(\Z)$ 
is factorizable on $\operatorname{SL}_2(\Z)$.
We have $E_4^3-E_6^2=1728\Delta$ (see~\cite{z}), which is factorizable
on $\Gm_0(2)$ (see~\ref{11.09.2015}).
Clearly, 
$E_4^3-tE_6^2$ is nonzero at $\infty$ for all $t\neq1$.
The valence formula 
for $\operatorname{SL}_2(\Z)$ implies
that for each $t\in\C^*\smallsetminus\{1\}$, 
$E_4^3-tE_6^2$ 
vanishes only at one point $z_t$ in a fundamental domain of
$\operatorname{SL}_2(\Z)$. 
Since $t$ is nonzero and since $E_4$ and $E_6$ have no common zeros,
neither  $E_4(z_t)$ nor $E_6(z_t)$ is zero. In particular, 
the stabilizer of 
$z_t$ in $\operatorname{SL}_2(\Z)$ (hence, also in $\Gm_0(2)$) 
is trivial. It follows that $E_4^3-tE_6^2$ only has a simple zero at $z_t$.
In particular, for each $t\in\C^*\smallsetminus\{1\}$, 
the modular form above 
is not 
factorizable on  $\operatorname{SL}_2(\Z)$.
Now, the valence formula (\ref{3.2Aug}) for $\Gm_0(2)$ 
implies that for each such $t$,\hspace*{1.2pt} 
$E_4^3-tE_6^2$ has three distinct zeros on $\Gm_0(2)\backslash\H$,
one of which is equivalent to $z_t$ under the action of $\Gm_0(2)$ on $\H$,
whereas the modular form
\begin{equation}
\e_2(z_t)^{24}\hspace{1pt}\frac{\e^{16}}{\e_2^8}-\e(z_t)^{24}\hspace{1pt}\frac{\e_2^{16}}{\e^8}
\end{equation}
of weight 4 on $\Gm_0(2)$ 
only has a simple zero at $z_t$ in $\Gm_0(2)\backslash\H$. It follows that for $t\in\C^*\smallsetminus\{1\}$, the modular form above
is a factor of $E_4^3-tE_6^2$ on 
$\Gm_0(2)$.
Thus, 
every modular form on $\operatorname{SL}_2(\Z)$ is factorizable on $\Gm_0(2)\cap\Gm_0(4)=\Gm_0(4)$.
However, since the smallest weight of which 
nonzero modular forms 
with the trivial multiplier system exist is 2, the modular form of weight~2 on $\Gm_0(N)$ defined by 
\begin{equation}
NE_2(Nz)-E_2(z)
\end{equation}
(see \cite{ds}) is irreducible for all  $N>1$.
It is not known whether for any level~$N$,
there exist any ``irreducible modular form'' of weight greater than~2.
On the contrary, 
 it follows from Theorem~\ref{10.09.2015} below (or from Corollary~2 in~\cite{B-three}) that
there exist irreducible holomorphic eta quotients (i.~e.
which are not products of other holomorphic 
eta quotients)
of arbitrarily large weights.
We shall 
also see some irreducibility criteria 
for
holomorphic eta quotients
in 
\cite{B-seven}.

\section{The results}
\label{26.08.2015.2} 
In order to state our main results,
first 
we introduce 
some notations. 
Denote the set of positive integers by $\N$.
For $N\in\N$, by ${{{\wp}}_{{N}}}$ we denote the set of prime divisors of $N$.
For 
a divisor $d$ of $N$, we say that $d$ \emph{exactly divides} $N$ and write $d\|N$ 
if $\gcd(d,N/d)=1$.
Below in Corollary~\ref{b11}, we provide an upper bound for the weights of the irreducible 
holomorphic eta quotients
of a given~level 
in~terms of
the function $\kappa:\bb{N}\rightarrow\bb{N}$ 
defined by
\begin{equation}\kappa(N)=\varphi\left(\rad(N)\right)\prod_
{\substack{p\in{{\wp}}_{{N}}\\p^n\|N
 }}
\big{(}(n-1)(p-1)+2\big{)},\label{SB1}\end{equation}
where $\varphi$ denotes Euler's totient function and
$\rad(N)$  denotes the product of the distinct prime divisors of $N$.
Let $\operatorname{d}:\N\rightarrow\N$ 
denote the divisor function.
Clearly, 
we have \mbox{ $\kappa(N)\leq \varphi(\rad(N))^2\hspace*{-2.5pt}\cdot\operatorname{d}(N)$.}
Below in Theorem~\ref{b1} (resp.~Corollary~\ref{b11}), we provide an upper bound $\varOmega(N)$ (resp.~$\varOmega_0(N)$)
for the number of the holomorphic eta quotients which are not factorizable on $\Gm_0(N)$
(resp.~irreducible holomorphic eta quotients of level~$N$).
We define the functions $\varOmega, \varOmega_0:\N\rightarrow\N$ by\hspace*{1.2pt} $\varOmega(1)=\varOmega_0(1)=1$ and for $N>1$,
\begin{align}
\varOmega(N)=&\prod_{\substack{p\in{{\wp}}_{{N}}\\p^n\|N
}}p^{2\operatorname{d}(N)}\Big(\frac{p^2-1}{p^4}\Big)^{\operatorname{d}(N/p^n)}
-\frac1{d(N)!}\prod_{\substack{
p\in{{\wp}}_{{N}}\\p^n\|N
}}\frac{(p^2-1)^{\operatorname{d}(N)}}{(p+1)^{2\operatorname{d}(N/p^n)}}
\label{mousekokhay}\\
&+2\Big(\hspace*{-1.5pt}\operatorname{d}(N)-\sum_{\substack{p\in{{\wp}}_{{N}}\\p^n\|N}}\operatorname{d}(N/p^n)\Big)-2^{\omega(N)}(\omega(N)-1)
\nonumber\end{align}
and
\begin{equation}\varOmega_0(N)=\varOmega(N)-2\operatorname{d}(N)+2^{\omega(N)}+1,
\label{mousekokhay1}\end{equation}
where $\omega(N)$ denotes the number of distinct prime divisors of $N$.
It is rather easy 
to show that 
$\varOmega(N)$ (and hence, also $\varOmega_0(N)$) is bounded above by
$\rad(N)^{2\operatorname{d}(N)}$
for all $N$ (see~Lemma~\ref{May 15, 2017} below).
We say that a holomorphic eta quotient $f$ is \emph{divisible} by a holomorphic eta quotient $g$
if $f/g$ is holomorphic.
We shall show that
\begin{theorem}For all $N\in\N$, 
the following assertions hold: \label{b1}
\begin{itemize}
\item[$(a)$] The weight of any holomorphic eta quotient on $\GN$ which is not factorizable on $\GN$\hspace{.6pt} is less than 
$\kappa(N)/2$.
\item[$(b)$] 
The number of nonconstant
 holomorphic eta quotients on $\GN$ which are not factorizable on $\GN$
 is 
bounded above by $\varOmega(N)$.
\item[$(c)$] 
 There are 
 at most $\varOmega_0(N)$
quasi-irreducible holomorphic eta quotients of level~$N$.
\end{itemize}
\end{theorem}

In particular, since any irreducible holomorphic eta quotient is quasi-irreducible, from the above theorem we conclude: 
\begin{co}For all $N\in\N$, 
the following assertions hold:
\begin{itemize}
\item[$(a)$] The weight of any irreducible holomorphic eta quotient of level $N$\hspace{.6pt} is 
less than $\kappa(N)/2$.
\item[$(b)$] The number of irreducible holomorphic eta quotients of level~$N$ is 
bounded above by
$\varOmega_0(N)$.\qed
\end{itemize}
 \label{b11}
\end{co}
 
In fact, $\kappa(N)/2$ is 
the 
smallest possible weight
for an eta quotient $f$ such that $f/g$ is holomorphic
for all 
holomorphic eta quotients $g$ which are 
not factorizable on $\Gm_0(N)$:
\begin{theorem}
For all $N\in\N$, there exists a 
holomorphic eta quotient $F_N$ of weight\hspace{1.5pt} $\kappa(N)/2$ on $\Gamma_0(N)$ such that 
a holomorphic eta quotient $h$ on $\Gamma_0(N)$ is divisible by $F_N$
if and only if  $h$ is divisible by all 
the holomorphic
eta quotients on $\GN$  
which are 
not factorizable on $\Gamma_0(N)$.
\label{17.10Sept}\end{theorem} 

 In the above 
 theorem,
 the uniqueness of the eta quotient $F_N$ readily follows from the claim. 
 We shall see 
 $F_N$ explicitly in (\ref{08.09.2015}).
 We recall that the Reducibility Conjecture (see Conjecture~$1'$ in \cite{B-three}) states: 
 \emph{Every quasi-irreducible holomorphic eta quotient is irreducible}.
 Since holomorphic eta quotients on $\GN$ which are not factorizable on $\GN$ are in particular
 quasi-irreducible,  it follows from the above theorem that 
 
 \begin{co}
If the Reducibility Conjecture $($Conjecture~$1'$ in \cite{B-three}$)$ holds, then for all $N\in\N$,
there exists a 
holomorphic eta quotient $F_N$ of weight\hspace{1.5pt} $\kappa(N)/2$ on  $\Gamma_0(N)$ such that 
a holomorphic eta quotient $h$ on $\Gamma_0(N)$ is divisible by $F_N$
if and only if  $h$ is divisible by all 
the irreducible holomorphic eta quotients on $\GN$. 
\label{17.10Sept.A}\qed\end{co}

We shall also show that
\begin{theorem}
For $N\in\N$ and for any divisor $t$ of $N/\rad(N)$, there 
is  an irreducible holomorphic eta quotient of 
weight\label{10.09.2015}
 $$\frac12\varphi(\rad(N))\hspace{1.5pt}\varphi(\rad(\gcd(t,N/t)))$$
 on $\GN$. In particular, for $t=N/\rad(N)$, 
there exists  an irreducible holomorphic eta quotient of  level $N$ and
of the weight as above.
\end{theorem} 
 
\section{Notations and the basic facts}
\label{26.08.2015.1}
For $N\in\N$, by $\D$ we denote the set of divisors of $N$.
For 
$X\in\ZD$, we define the eta quotient $\e^X$
   by
   \begin{equation}\label{3Jan15.1}
    \e^X:=\displaystyle{\prod_{d\in\D}\eta_d^{X_d}},
 \end{equation}
where $X_d$ is the 
value of $X$ 
at $d\in\D$ whereas $\e_d$ denotes the rescaling of $\e$ by $d$.
Clearly, the level of $\e^X$ divides $N$. In other words, $\e^X$ transforms like a modular form on $\GN$. 
We define the summatory function $\sigma:\ZD\rightarrow\Z$ by \begin{equation}
\sigma(X):=\sum_{d\in\D}X_d.
\label{30.08.2015.A}\end{equation}
Since $\e$ is of weight $1/2$,  
the weight of $\e^X$ is  $\sigma(X)/2$ for all $X\in\ZD$.

Recall that an eta quotient $f$ on $\GN$ is holomorphic
if it 
does not have any poles at the cusps of $\GN$. Under the action of $\GN$ on $\PQ$
by M\"obius transformation, for 
$a,b\in\Z$ with $\gcd(a,b)=1$,
we have
\begin{equation}
[a:b]\hspace{.1cm}
{{{{\sim}}}}_{\hspace*{-.05cm}{{{{\GN}}}}}
\hspace{.08cm}[a':\gcd(N,b)]\label{19.05.2015} 
\end{equation}
for some $a'\in\Z$ which is coprime to $\gcd(N,b)$ (see \cite{ds}).
We identify $\PQ$ with $\Q\cup\{\infty\}$ via the canonical bijection that maps $[\alpha:\lambda]$ to 
$\alpha/\lambda$ if $\lambda\neq0$ and to $\infty$ if $\lambda=0$. 
For $s\in\Q\cup\{\infty\}$ and a weakly holomorphic modular form $f$ on $\GN$, the order of $f$ at the cusp $s$ of $\GN$ is 
the exponent of 
 \hyperlink{queue}{$q^{{1}/{w_s}}$} 
occurring with the first nonzero coefficient in the 
$q$-expansion of $f$
at the cusp $s$,
where $w_s$ is the width of the cusp $s$ (see \cite{ds,ra}).
The following is 
the set of the equivalence classes of the cusps of $\GN$ 
(see \cite{ds,ymart}):
\begin{equation}
\cl{S}_N:=\Big{\{}\frac{a}{t}\in\Q\hspace{2.5pt}{\di}\hspace{2.5pt}t\in\cl{D}_N,\hspace{2pt} a\in\bb{Z}, 
 \hspace{2pt}\gcd(a,t)=1\Big{\}}/\sim\hspace{1.5pt},
\end{equation}
where $\dfrac{a}{t}\sim\dfrac{b}{t}$ if and only if $a\equiv b\pmod{\gcd(t,N/t)}$.
For $d\in\D$ and for $s=\dfrac{a}t\in\cl{S}_N$ with $\gcd(a,t)=1$, we have
\begin{equation}
 \ord_s(\e_d\hspace{1pt};\GN)= \frac{N\cdot\gcd(d,t)^2}{24\cdot d\cdot\gcd(t^2,N)}\in\frac1{24}\N 
\label{26.04.2015}\end{equation}
(see 
\cite{ymart}). 
It is easy to check the above inclusion  
when $N$ is a prime power. 
The general case 
follows by multiplicativity (see \ref{27.04.2015} and \ref{13May}).
It follows that for all $X\in\ZD$, we have
\begin{equation}
  \ord_s(\e^X\hspace{1pt};\GN)= \frac1{24}\sum_{\substack{d\in\D}}\frac{N\cdot\gcd(d,t)^2}{d\cdot\gcd(t^2,N)}X_d\hspace{1.5pt}. 
\label{27.04}\end{equation}
 In particular, 
that implies
\begin{equation}
 \ord_{a/t}(\e^X\hspace{1pt};\GN)=\ord_{1/t}(\e^X\hspace{1pt};\GN)
\label{27.04.2015.1}\end{equation}
for all $t\in\D$ and for all the $\varphi(\gcd(t,N/t))$ inequivalent cusps of $\GN$
represented by rational numbers
of the form $\dfrac{a}{t}\in\cl{S}_N$ with $\gcd(a,t)=1$.
Let $\psi(N)$ denote the index of 
$\GN$ in $\operatorname{SL}_2(\Z)$. Then $\psi:\N\rightarrow\N$ is 
given 
by
\begin{equation}
\psi(N):= 
N\cdot\hspace*{-.2cm}\prod_{\substack{p|N\\\text{$p$ prime}}}\hspace*{-.2cm}\left(1+\frac{1}{p}\right)
\label{14.2Sept}\end{equation}
(see \cite{ds}).
The \emph{valence formula} for $\GN$ (see \cite{be,ra}) states:
\begin{equation}
\sum_{P\in\hspace*{1pt}\GN\backslash\H}\frac1{n_P}\cdot\ord_P(f)
\hspace{2pt}+\sum_{s\in\hspace{.5pt}\cl{S}_N}\ord_{s}(f\hspace{1pt};\GN) 
=\frac{k\cdot\psi(N)}{24}\hs,\label{3.2Aug}
\end{equation}
where $k\in\Z$, $f:\H\rightarrow\C$ is a meromorphic function that transforms like a modular forms of weight
$k/2$ on $\GN$ which is also meromorphic at the cusps of $\GN$ and $n_P$ is the number of elements in the 
stabilizer of $P$ in the group $\GN/\{\pm I\}$, where $I\in\operatorname{SL}_2(\Z)$ denotes the identity matrix.
In particular, if $f$ is an eta quotient, then
from (\ref{3.2Aug}) we obtain
\begin{equation}
\sum_{s\in\hspace{.5pt}\cl{S}_N}\ord_{s}(f\hspace{1pt};\GN) 
=\frac{k\cdot\psi(N)}{24}\hs,
\label{27.04.2015.2}\end{equation}
because  eta quotients do not have poles or zeros on $\H$. 
it follows from  (\ref{27.04.2015.2})
and from 
 (\ref{27.04.2015.1}) that for an eta quotient $f$ of weight $k/2$ on $\GN$, the valence formula
 further reduces to
\begin{equation}
\sum_{t\hspace{.5pt}\idi N} 
\varphi(\gcd(t,N/t))\cdot\ord_{1/t}(f\hspace{1pt};\GN) 
=\frac{k\cdot\psi(N)}{24}\hspace{1pt}.
\label{27.04.2015.3}
\end{equation}
Since $\ord_{1/t}(f\hspace{1pt};\GN)\in\frac{1}{24}\Z$ (see \ref{26.04.2015}), from
(\ref{27.04.2015.3}) it follows that of any particular weight, there are only finitely many holomorphic eta quotients
on $\GN$. More precisely, the number of holomorphic eta quotients of weight $k/2$
on $\GN$ is at most the number of solutions of the following equation
\begin{equation}
 \sum_{t\in\D}{\varphi(\gcd(t,N/t))}\cdot x_t=k\cdot\psi(N)
\label{27.04.2015.3Z}\end{equation}
in nonnegative integers $x_t$. 

We define the \emph{\hypertarget{om}{order map}}
$\cl{O}_N:\ZD\rightarrow\frac1{24}\ZD$ of level $N$ as the map which sends $X\in\ZD$ to the ordered set of
orders of the eta quotient $\e^X$ at the cusps $\{1/t\}_{t\in\D}$ of $\GN$. Also, we define the
\emph
{order matrix} $A_N\in\Z^{\D\times\D}$ of level $N$ by 
\begin{equation}
 A_N(t,d):=24\cdot\ord_{1/t}(\e_d\hspace{1pt};\GN) 
\label{27.04.2015}\end{equation}
for all $t,d\in\D$.
 For example, for a prime power $p^n$, 
we have
\begin{equation}
A_{p^n}=\begin{pmatrix}
 \vspace{5.8pt}p^n &p^{n-1} &p^{n-2} 
&\cdots  &p &1\\
\vspace{5.8pt} p^{n-2} &p^{n-1} &p^{n-2} 
&\cdots  &p &1\\
 \vspace{5.8pt}p^{n-4} &p^{n-3} &p^{n-2} 
&\cdots  &p &1\\
\vspace{5.8pt} \vdots &\vdots &\vdots 
&\cdots &\vdots &\vdots\\
\vspace{5.8pt} 1 &p &p^2 
&\cdots &p^{n-1} &p^{n-2}\\
 1 &p &p^2 
&\cdots  &p^{n-1} &p^n
\end{pmatrix}.
 \label{23July}
\end{equation}
By linearity of the order map, we have 
\begin{equation}
\cl{O}_N(X)=\frac1{24}\cdot A_NX\hspace{1.5pt}. 
\label{28.04}\end{equation}
For $r\in\N$, if $Y,Y'\in\Z^{\D^{\hspace{.5pt}r}}$ is 
such that $Y-Y'$ is
nonnegative at each element of $\D^{\hspace{.5pt}r}$, then
we write $Y\geq Y'$. 
In particular, for $X\in\ZD$, the eta quotient $\e^X$ is holomorphic if and only if $A_NX\geq0$.

From $(\ref{27.04.2015})$ and $(\ref{26.04.2015})$, we note that $A_N(t,d)$ is multiplicative in $N, t$ and $d$.
Hence, it follows that
\begin{equation}
 A_N=\bigotimes_{\substack{ p^n\|N\\\text{$p$ prime}}}A_{p^n},
\label{13May}\end{equation}
where by \hspace{1.5pt}$\otimes$\hspace{.2pt},  we denote the Kronecker product of matrices.\footnote{Kronecker product of matrices is
not commutative. However, 
since any given ordering of the primes dividing $N$ induces a lexicographic ordering on $\cl{D}_N$ 
with which the entries of $A_N$ are indexed, 
Equation (\ref{13May}) makes sense for all possible 
orderings of the primes dividing $N$.} 

It is easy to verify that for a prime power $p^n$, 
the matrix $A_{p^n}$ is invertible with the tridiagonal inverse: 
\begin{equation}
A_{p^n}^{-1}=
\frac{1}{p^{n}\cdot(p-\frac1p)}
\begin{pmatrix}
\hspace{6pt}p &\hspace*{-6pt}-p &  
& &  & \\
\vspace{5pt}\hspace*{-1pt}-1 & \hspace*{-2pt}p^2+1 &\hspace*{-3pt}-p^2 
& &\textnormal{\Huge 0} & \\
\vspace{5pt}&\hspace*{-6pt}-p & \hspace*{-4pt}p\cdot(p^2+1) &\hspace*{-2pt}-p^3 
 &  &  \\
 &  &\hspace*{-5pt}\ddots 
 &\hspace*{-6pt}\ddots & \ddots & \\
\vspace{5pt}\hspace{2pt} &\textnormal{\Huge 0}  &  
 &\hspace*{-7pt}-p^2 & \hspace*{-4pt}p^2+1 &\hspace*{-4pt}-1\hspace{2pt}\\
\hspace{2pt} &  & 
&  &\hspace*{-6pt}-p & p\hspace{2pt}
\end{pmatrix},
 \label{r1}\end{equation}
where 
for each positive integer $j<n$, 
the nonzero entries of the column  $A_{p^n}^{-1}(\un\hs,p^j)$ are the same as 
those of the column $A_{p^n}^{-1}(\un\hs,p)$ shifted down by $j-1$ entries and
multiplied with $p^{\min\{j-1,n-j-1\}}$. 
More precisely,
\begin{align}
p^{n}\cdot(p-\frac1p)&
\cdot A_{p^n}^{-1}(p^i,p^j)\nonumber=\\&\begin{cases}
                        \ \hs\ p&\text{if $i=j=0$ or $i=j=n$}\\
                        -p^{\min\{j,n-j\}}&\text{if $|i-j|=1$}\\
                        \ \hs\ p^{\min\{j-1,n-j-1\}}\cdot(p^2+1)&\text{if $0<i=j<n$}\\
                        \ \hspace{1.2pt}\ 0&\text{otherwise.}\label{r11}
\end{cases}
\end{align}
For general $N$, the invertibility of the matrix $A_N$ now 
follows by (\ref{13May}).
Hence, any eta quotient on $\GN$ is uniquely determined by its orders at the set of the cusps
$\{1/t\}_{\substack{\ \\ t\in\D}}$ of $\GN$. In particular, for distinct $X,X'\in\ZD$, we have $\e^X\neq\e^{X'}$. 
The last statement 
is also implied by the uniqueness of $q$-series expansion:
Let $\e^{\widehat{X}}$ and $\e^{\widehat{X}'}$
be the \emph{eta products} (i.~e.  $\widehat{X}, \widehat{X}'\geq0$)
obtained by multiplying $\e^X$ and $\e^{X'}$ with a common denominator. The claim follows by induction on the weight of $\e^{\widehat{X}}$
(or equivalently, the weight of $\e^{\widehat{X}'}$)
when we compare
the corresponding 
first two exponents of $q$
occurring in the $q$-series expansions of 
$\e^{\widehat{X}}$ and $\e^{\widehat{X}'}$.

\section{
The finiteness}
\label{19Sept}
In this section, we prove the finiteness of irreducible holomorphic eta quotients of a given level 
(see 
the corollary 
of Theorem~\ref{b1}).

Let $A_N$ be the order matrix of level $N$ (see \ref{27.04.2015}).
From the invertibility of $A_N$, it follows trivially that for each $t\in\D$,
there is an eta quotient which vanishes nowhere except at the cusps
$a/t$ of $\GN$ for all integers $a$ which are coprime to $t$
(see Corollary~1.42 in the Preliminaries of \cite{B-two})
\footnote{The invertibility of the order matrix (and hence, the existence of
such eta quotients) has been known classically.
For example, see Satz~8 in \cite{c}, Proposition~3.2 in \cite{b}, the proof of Theorem~3 in \cite{km} or the proof of Theorem~2 in \cite{rw}.}. 
Let $B_N\in\Z^{\D\times\D}$ be the matrix whose columns 
are 
made of the exponents of these eta quotients. 
A little more precise description of $B_N$ is 
as follows: 
Since all the entries of $A_N^{-1}$
are rational (see \ref{13May}, \ref{r1}), for each $t\in\D$, there exists a smallest positive integer $m_{{t,N}}$ such that 
$m_{{t,N}}\cdot A_N^{-1}(\un\hs,t)$
has integer entries, where $A_N^{-1}(\un\hs,t)$ denotes the 
column of $A_N$ indexed by $t\in\D$.
We define $B_N\in\Z^{\D\times\D}$ by
\begin{equation}
B_N(\un\hs,t):=m_{{t,N}}\cdot A_N^{-1}(\un\hs,t)\hspace{4pt}\text{for all $t\in\D$.}
\label{28.08.2015}\end{equation}
Clearly, $B_N$ is invertible over $\Q$. 
Recall that 
for $X\in\ZD$, $\eta^X$ 
is holomorphic if and only if $A_NX\geq 0$  (see \ref{28.04}).
We define the eta quotient $F_N$ by
 \begin{equation}
  F_N:=\prod_{t\in\D}\e^{
  B_N(\un\hs,t)}.
 \label{08.09.2015}\end{equation}
The lemma below follows immediately:
\begin{lem}For $N\in\N$, let $F_N$ be as defined above. Then
for $X\in\ZD$, both of the eta quotients $f:=\e^X$ and $F_N/f$ are holomorphic if and only if
$X\in B_N\cdot[0,1]^{\D}$. \label{08.09.2015A} 
\qed
\end{lem}
Let $X\in\ZD\smallsetminus\{0\}$ be such that $\e^X$ is a holomorphic eta quotient which is not factorizable on $\GN$.
Define $Y\in\ZD$ by $Y:=B_N^{-1}X$. Suppose, for some $t\in\D$, we have $Y_{t}\geq1$. 
Then $\e^X$ is divisible by the nonconstant holomorphic eta quotient $\e^{B_N(\un\hs,t)}$.
Since $\e^X$ is not 
factorizable on $\GN$,
we conclude that 
$X=B_N(\un\hs,t)$. 
Thus, we have proved that

\begin{lem}For $N\in\N$, let $B_N\in\Z^{\D\times\D}$ be as defined in $(\ref{28.08.2015})$. 
For $X\in\ZD$, if $\e^X$ is a holomorphic eta quotient which is not factorizable on $\GN$,
then either $X\in B_N\cdot[0,1)^{\D}$ or $X=B_N(\un\hs,t)$ for some $t\in\D$.\qed
\label{b2}\end{lem}

Since for $N\in\N$, there are only finitely many lattice points in the bounded 
polytope $B_N\cdot[0,1)^{\D}$\hspace*{-2pt},
from Lemma~\ref{b2}\hspace{1pt} it follows that there are only finitely many holomorphic eta quotients on $\GN$ which are not 
factorizable on $\GN$\hspace{1pt} (e.~g.  the irreducible holomorphic eta quotients whose levels divide $N$).

\ \newline
\textit{Proof of Theorem \ref{b1}.(a).} 
Let 
$f$ be 
a holomorphic eta quotient on $\GN$ which is not factorizable on $\GN$.
From the above lemma, we see that the weight of $f$ 
is at most equal to the maximum value of $ 
\sigma(X)/2
$, where $X$ varies over $B_N\cdot[0,1]^{\D}$ and $\sigma$ is as defined in (\ref{30.08.2015.A}).
Since for all $t\in\D$, the sum of all the entries in the column $B_N(\un\hs,t)$ 
of $B_N$ is positive (see \ref{30.08.2015}), 
it follows that
\begin{equation}
\max_{X\in B_N\cdot[0,1]^{\D}}\sigma(X)=\sum_{t\in\D}\sigma(B_N(\un\hs,t)).
\nonumber\end{equation}
Hence, it suffices to show that
\begin{equation}
\kappa(N)=\sum_{d\in\D}\sigma(B_N(\un\hs,t)).
\label{30.08.2015.B}\end{equation}
Since for $N\in\N$ and $t\in\D$, all the entries of the columns $A_N^{-1}(\un\hs,t)$ 
are multiplicative in $N$ and $t$\hspace{1.3pt} (see \ref{13May}),
so is 
the smallest positive integer $m_{{t,N}}$ such that $m_{{t,N}}\cdot A_N^{-1}(\un\hs,t)\in\ZD$
(see Lemma~4 in \cite{B-three}).
Hence, from the multiplicativity of $A_N^{-1}(d,t)$ in $N$, $d$ and $t$  (see \ref{13May}),
it follows that $B_N(d,t)$ (see \ref{28.08.2015}) is also multiplicative in $N$, $d$ and $t$. 
That implies:
\begin{equation}
 B_N=\bigotimes_{\substack{ \hspace*{3pt}p\in{{\wp}}_{{N}}\\p^n\|N
 }}
 B_{p^n},
\label{28.08.2015.1}\end{equation}
where ${{{\wp}}_{{N}}}$ denotes the set of prime divisors of $N$.
For a prime $p$, 
from (\ref{28.08.2015}) and 
(\ref{r1}), we have 
 \begin{equation}
 B_{p^n}=
\begin{pmatrix}
\hspace{6pt}p &\hspace*{-6pt}-p &  
& &  & \\
\vspace{5pt}\hspace*{-1pt}-1 & \hspace*{-2pt}p^2+1 &\hspace*{-3pt}-p 
& &\textnormal{\Huge 0} & \\
\vspace{5pt}&\hspace*{-6pt}-p & \hspace*{-4pt}p^2+1 &\hspace*{-2pt}-p 
 &  &  \\
 &  &\hspace*{-5pt}\ddots 
 &\hspace*{-6pt}\ddots & \ddots & \\
\vspace{5pt}\hspace{2pt} &\textnormal{\Huge 0}  &  
 &\hspace*{-7pt}-p & \hspace*{-4pt}p^2+1 &\hspace*{-4pt}-1\hspace{2pt}\\
\hspace{2pt} &  & 
&  &\hspace*{-6pt}-p & p\hspace{2pt}
\end{pmatrix}.
\label{equ1}
\end{equation}
Summing up the entries of each 
column of $B_{p^n}$, we get:
\begin{equation}
\sigma(B_{p^n}(\un\hs,{p^j}))=
\begin{cases}
\hspace{4.4pt}p-1&\text{if $j=0$ or $j=n$}\\
(p-1)^2&\text{otherwise.} 
\end{cases}
\label{equ2}
\end{equation}
Since (\ref{28.08.2015.1}) implies that  
\begin{equation}
B_N(\un\hs,t)=\hspace*{-.1cm}\bigotimes_{\substack{ \hspace*{7pt}p\in{{\wp}}_{{N}}\\p^j\|t
 }}
 B_{p^n}(\un\hs,{p^j})\hspace{4.5pt}\text{for all $d\in\D$},
\label{31.08.2015} \end{equation}
from (\ref{equ2}) we get:
\begin{align}
\sigma(B_N(\un\hs,t))&=\prod_{\substack{ \hspace*{7pt}p\in{{\wp}}_{{N}}\\p^j\|t
 }}
 \sigma(B_{p^n}(\un\hs,{p^j}))
\label{30.08.2015}\\ 
 &=
 \Big{(}\hspace*{-.45cm}\prod_{\substack{ \hspace*{6pt}p\in{{\wp}}_{{N}}\\p\nidi\gcd(t,N/t)
 }}\hspace*{-.3cm}(p-1)\Big{)}\hspace{2.5pt}\cdot\hspace*{-.45cm}\prod_{\substack{ \hspace*{6pt}p\in{{\wp}}_{{N}}\\p\idi\gcd(t,N/t)
 }}\hspace*{-.3cm}(p-1)^2\nonumber\\
 \nonumber\\
 &=\varphi(\rad(N))\cdot\varphi(\rad(\gcd(t,N/t))). \nonumber
\end{align}
Since 
$ 
\varphi(\rad(\gcd(t,N/t)))$ 
is multiplicative in 
$t\in\D$, 
the summatory function
$N\mapsto\displaystyle{\sum_{t\in\D}\varphi(\rad(\gcd(t,N/t)))}$ is 
multiplicative in $N$. 
So, 
\begin{align}
\sum_{t\in\D}\varphi(\rad(\gcd(t,N/t)))&=\prod_{\substack{ \hspace*{3pt}p\in{{\wp}}_{{N}}\\p^n\|N
 }}\sum_{j=0}^n\varphi(\rad(p^{\min\{j,n-j\}}))\label{30.08.2015.C}\\
 &=\prod_{\substack{ \hspace*{3pt}p\in{{\wp}}_{{N}}\\p^n\|N
 }}((n-1)(p-1)+2).\nonumber
\end{align}
Now, (\ref{30.08.2015.B}) follows from (\ref{30.08.2015}) and (\ref{30.08.2015.C}). 
\vspace{2pt}

The only $X\in B_N\cdot[0,1]^{\D}$ with $\sigma(X)=\kappa(N)$ is $X=\sum_{t\in\D}\hspace*{-.1cm}B_N(\un\hs,t)$.
Since $N>1$,  it follows trivially 
from Lemma~\ref{b2}, that for such an $X$, the holomorphic eta quotient
$\e^X$ is factorizable on $\GN$. 
\qed

\ \newline
\textit{Proof of Theorem \ref{b1}.(b).}  
In Lemma~\ref{b2}, we saw that each holomorphic eta quotient which is not factorizable on $\Gm_0(N)$
correspond either to a column of $B_N$ or to a 
lattice point in the fundamental parallelepiped $B_N\cdot[0,1)^{\D}$. 
Clearly, the number of the columns of 
$B_N$ is $\operatorname{d}(N)$. In Lemma~\ref{May 14,2017} below, we show that the number of
lattice points in a fundamental parallelepiped of~$B_{N}$ is
\begin{equation}
\varOmega'(N):=\prod_{\substack{p\in{{\wp}}_{{N}}\\p^n\|N}}p^{2\operatorname{d}(N)}\Big(\frac{p^2-1}{p^4}\Big)^{\operatorname{d}(N/p^n)}.
\label{omgdash}\end{equation}
However, there also exist lattice points in the fundamental parallelepiped $B_N\cdot[0,1)^{\D}$
which correspond to some holomorphic eta quotients which are factorizable on $\Gm_0(N)$. 
For example, 
if $X$ is a lattice point 
outside the unit 
sphere in $\R^{\operatorname{d}(N)}$ 
such that all its entries are nonnegative,
then $\e^X$ is clearly factorizable on~$\Gm_0(N)$.
In Lemma~\ref{May 14,2017.A} below, we show that the number of such
lattice points in $B_N\cdot[0,1)^{\D}$ 
is at least
\begin{align}
\varOmega''(N):=&\frac1{\operatorname{d}(N)!}\prod_{\substack{p^n\parallel N\\p \text{ prime}}}\frac{(p^2-1)^{\operatorname{d}(N)}}{(p+1)^{2\operatorname{d}(N/p^n)}}
+2\sum_{\substack{p\in{{\wp}}_{{N}}\\p^n\|N}}\operatorname{d}(N/p^n)\label{omg2dash}\\
&-2^{\omega(N)}(\omega(N)-1)-\operatorname{d}(N),\nonumber
\end{align}
where $\omega(N)$ denotes the number of distinct prime divisors of $N$.
Hence, we conclude that the number of holomorphic eta quotients which are not factorizable on $\Gm_0(N)$
is bounded above by 
\begin{equation}
 \varOmega(N)=\operatorname{d}(N)+\varOmega'(N)-\varOmega''(N).
\end{equation}
{\flushright\qed}

Now, we prove the lemmas which were necessary in the above proof:
\begin{lem}
There are exactly $\varOmega'(N)$ lattice points in a fundamental parallelepiped of $B_N$,
where $\varOmega':\N\rightarrow\N$ is as defined in $(\ref{omgdash})$.
\label{May 14,2017}
\end{lem}
\begin{proof}
Since the number of integer points in a fundamental parallelepiped of a nonsingular integer matrix is
equal to the volume of the parallelepiped (see Theorem~2 in \cite{bar}),
it suffices to show that the determinant of $B_N$ is~$\varOmega'(N)$.
Indeed, for a prime number $p$ and a positive integer $n$, we have $\det(B_{p^n})=\varOmega'(p^n)$
 which follows trivially after transforming $B_{p^n}$ (see~\ref{equ1})
 to the following 
 matrix through elementary column operations 
$$\begin{tikzpicture} [baseline=(current bounding box.center)]
\matrix (m) [matrix of math nodes,nodes in empty cells,right delimiter={)},left delimiter={(}] {
p &                                  & & & & & &    &                    &                                                             \\
\hspace*{-10pt}-1   &p^2                                 & & & & & &    &                    &                           \\
   &\hspace*{-11pt}-p  &p^2 &\hspace{5pt}\color{white}{p^2} & & & &    &                    &                         &                                    \\
       &  & &\hspace*{-45pt}-p &\hspace*{-25pt}p^2&\hspace*{-10pt}\color{white}{p^2}& & & & \textnormal{\Huge 0}\\
       &&&&\hspace*{-30pt}-p
       &\hspace*{-5pt}p^2\hspace*{-2.5pt}-\hspace*{-2.5pt}1 &\hspace*{-1pt}-p & &    &                                                                                 \\
          &  &  & & & &\hspace*{7pt}p^2 &-p &\color{white}{p^2}    &                    &                                                            \\
  \textnormal{\Huge 0}&                                  & & & & & &\hspace*{7pt}p^2 &-p&\hspace{7pt}\color{white}{p^2}                                    \\
&                                 & & & & & & &\hspace*{7pt}p^2 &\hspace*{-8pt}-1                                    \\
              &                                 & & & & & & & &\hspace*{-10pt}p                                    \\};
       \hspace*{-26.5pt}\draw[decorate sep={.75pt}{3.75mm},fill] (m-3-3)-- (m-4-5);
        \draw[decorate sep={.75pt}{3mm},fill](m-3-4)-- (m-4-6);
\draw[decorate sep={.75pt}{2mm},fill](m-6-8)-- (m-7-9);
\draw[decorate sep={.75pt}{1.573mm},fill](m-6-9)-- (m-7-10);
\end{tikzpicture}
$$
and 
from the fact that for square matrices $A$ and $D$, we have 
$$\det\begin{pmatrix}A & B\\0 & D\end{pmatrix}=\det(A)\det(D).$$ 
Since
for any two matrices $A_{m\times m}$ and $B_{n\times n}$,
\begin{equation}
 \det(A\otimes B)=\det(A)^n\det(B)^m
\label{May 18, 2017}\end{equation}
(see~\cite{a}), the general case now follows 
by induction on the number of prime divisors of $N$
(see \ref{28.08.2015.1}), 
\end{proof}

\begin{lem}Let $\varOmega'':\N\rightarrow\N$ be as defined in $(\ref{omg2dash})$.
In the fundamental parallelepiped $B_N\cdot[0,1)^{\D}$,
there are at least $\varOmega''(N)$ lattice points with
nonnegative coordinates, 
none of which lies on the unit sphere in $\R^{\D}$.
\label{May 14,2017.A}
\end{lem}
\begin{proof}
From (\ref{28.08.2015}), it follows that $B_N$ is invertible for all $N\in\N$.
For $n\in\N$ and a prime $p$, 
the matrix
$B_{p^n}$ is as 
in (\ref{equ1}).
It is easy to verify that 
\begin{equation}
B_{p^n}^{-1}=\frac{1}{p^{n}\cdot(p-\frac1p)}\begin{pmatrix}
 \vspace{5.8pt}p^n &p^{n-1} &p^{n-2} 
&\cdots  &p &1\\
\vspace{5.8pt} p^{n-2} &p^{n-1} &p^{n-2} 
&\cdots  &p &1\\
 \vspace{5.8pt}p^{n-3} &p^{n-2} &p^{n-1} 
&\cdots  &p^2 &p\\
 \vspace{5.8pt}p^{n-4} &p^{n-3} &p^{n-2} 
&\cdots  &p^3 &p^2\\
\vspace{5.8pt} \vdots &\vdots &\vdots 
&\cdots &\vdots &\vdots\\
 \vspace{5.8pt}p &p^{2} &p^{3} 
&\cdots  &p^{n-2} &p^{n-3}\\
\vspace{5.8pt} 1 &p &p^2 
&\cdots &p^{n-1} &p^{n-2}\\
 1 &p &p^2 
&\cdots  &p^{n-1} &p^n
\end{pmatrix}.
\label{May 16, 2017}\end{equation}
Clearly, the axes-intercepts of the fundamental parallelepiped $B_N\cdot[0,1]^{\D}$
is given by the reciprocals of the diagonal entries of $B_{N}^{-1}$.
Hence, from~(\ref{28.08.2015.1}) and (\ref{May 16, 2017}), it follows that 
the coordinates of the furthest points in $B_N\cdot[0,1]^{\D}$ 
on the axes of
$\R^{\D}$ are given by the columns of the matrix
\begin{equation}
 \bigotimes_{\substack{ \hspace*{3pt}p\in{{\wp}}_{{N}}\\p^n\|N
 }}\begin{pmatrix}
\hspace{6pt}p-\frac1p & &  
& &  & \\
\vspace{5pt}& \hspace*{-2pt}p^2-1 &\hspace*{-3pt} 
& &\textnormal{\Huge 0} & \\
\vspace{5pt}&\hspace*{-6pt} & \hspace*{-4pt}p^2-1 &\hspace*{-2pt}
 &  &  \\
 &  &\hspace*{-5pt} 
 &\hspace*{-6pt}\ddots &  & \\
\vspace{5pt}\hspace{2pt} &\textnormal{\Huge 0}  &  
 &\hspace*{-7pt} & \hspace*{-4pt}p^2-1 &\hspace*{-4pt}\hspace{2pt}\\
\hspace{2pt} &  & 
&  &\hspace*{-6pt} & p-\frac1p
\label{May 18, 2017.A}\end{pmatrix}. 
\end{equation}
In particular, 
$B_N\cdot[0,1]^{\D}$ contains 
the simplex $S_N$ which is the convex hull of
the origin and the points in $\R^{\D}$ 
whose coordinates are
given by the columns of the following matrix 
 \begin{equation}
 C_N=\bigotimes_{\substack{ \hspace*{3pt}p\in{{\wp}}_{{N}}\\p^n\|N
 }}\begin{pmatrix}
\hspace{6pt}p-1 & &  
& &  & \\
\vspace{5pt}& \hspace*{-2pt}p^2-1 &\hspace*{-3pt} 
& &\textnormal{\Huge 0} & \\
\vspace{5pt}&\hspace*{-6pt} & \hspace*{-4pt}p^2-1 &\hspace*{-2pt}
 &  &  \\
 &  &\hspace*{-5pt} 
 &\hspace*{-6pt}\ddots &  & \\
\vspace{5pt}\hspace{2pt} &\textnormal{\Huge 0}  &  
 &\hspace*{-7pt} & \hspace*{-4pt}p^2-1 &\hspace*{-4pt}\hspace{2pt}\\
\hspace{2pt} &  & 
&  &\hspace*{-6pt} & p-1
\label{May 18, 2017.B}\end{pmatrix}. 
\end{equation}
The number of lattice points in the $\operatorname{d}(N)$-dimensional rectangular parallelepiped $P_N:=C_N\cdot[0,1)^{\D}$
is clearly the same as its volume. i.~e. $\det(C_N)$.
Since the ratio of the volumes of $S_N$ and $P_N$ is $1/d(N)!$ (see~\cite{ps}), the simplex $S_N$
contains at least
\begin{equation}
\frac{(\det C_N)}{d(N)!} 
\label{May 19, 2017}\end{equation}
lattice points excluding 
all the vertices of $S_N$ except the origin. 
From (\ref{May 18, 2017.A}) and (\ref{May 18, 2017.B}), it follows that 
for $t\in\D$, 
$\gcd(t,N/t)$ is divisible by $\rad(N)$ if and only if
\begin{equation}
C_N(\un\hs,t)\in B_N\cdot[0,1]^{\D}\smallsetminus B_N\cdot[0,1)^{\D}.
\end{equation}
In other words, the number of nonzero vertices of $S_N$ which are contained in $B_N\cdot[0,1)^{\D}$
is the same as the number of $t\in\D$ such that $\rad(N)$
does not divide 
$\gcd(t,N/t)$.
It is easy to show that the number of such divisors of $N$
is 
\begin{equation}
2\sum_{\substack{p\in{{\wp}}_{{N}}\\p^n\|N}}\operatorname{d}(N/p^n)-2^{\omega(N)}(\omega(N)-1),
\label{May 19, 2017.A}\end{equation}
where $\omega(N)$ denotes the number of distinct prime divisors of $N$.
Again, from (\ref{May 18, 2017.B}) and (\ref{May 18, 2017}), it follows that
\begin{equation}
\det(C_N)=\prod_{\substack{p^n\parallel N\\p \text{ prime}}}\frac{(p^2-1)^{\operatorname{d}(N)}}{(p+1)^{2\operatorname{d}(N/p^n)}}
\label{May 19, 2017.B}\end{equation}
From (\ref{omg2dash}), (\ref{May 19, 2017}), (\ref{May 19, 2017.A}) and (\ref{May 19, 2017.B}), 
we obtain that there are at least 
$\varOmega''(N)+\operatorname{d}(N)$ lattice
points in $B_N\cdot[0,1)^{\D}$ with nonnegative coordinates.
However, exactly $\operatorname{d}(N)$ among these points 
lie on intersections of the unit sphere with the axes of $\R^{\D}$. 
\end{proof}

\ \newline
\textit{Proof of Theorem \ref{b1}.(c).} 
From Lemma~\ref{b2}, we recall that
for $X\in\ZD$, if $\e^X$ is a holomorphic eta quotient which is not factorizable on $\GN$,
then either $X\in B_N\cdot[0,1)^{\D}$ or $X=B_N(\un\hs,t)$ for some $t\in\D$.
The parallelepiped $B_N\cdot[0,1)^{\D}$ contains $\operatorname{d}(N)$ points which lie
at the intersections of the unit sphere with the axes of $\R^{\D}$. These points
corresponds to the rescalings of $\e$ by the divisors of $N$. In particular, these
are eta quotients of weight~$1/2$. So, each of these $\operatorname{d}(N)$ rescalings of $\e$ are irreducible,
whereas only one of them, viz.~$\e_N$ is of level $N$.

Next, we count the number of eta quotients of the form  $\e^{B_{N}(\un\hs,\hs t)}$
which are of level $N$. For a prime $p$,
from (\ref{equ1}) we see that the eta quotient $\e^{B_{p^n}(\un\hs,\hs p^j)}$ is of level $p^n$ if and only if $j\geq n-1$. 
Hence, from (\ref{31.08.2015}) it follows that for $N\in\N$, the eta quotient $\e^{B_{N}(\un\hs,\hs t)}$ is of level $N$
if and only if for each prime divisor $p$ of $N$, we have
$p^{n-1}\idi t$, where $n\in\N$ is such that $p^n\|N$. It is trivial to note that
the number of such divisors $t$ of $N$ is $2^{\omega(N)}$, where $\omega(N)$ denotes 
the number of prime divisors of $N$. 
Thus, among the $\operatorname{d}(N)$ columns of
$B_N$, only $2^{\omega(N)}$ correspond to eta quotients of level $N$.\qed

\

In the following, we provide a rather 
uncomplicated function which
dominates $\varOmega(N)$ for all $N$:

\begin{lem}
Let $\varOmega:\N\rightarrow\N$ be as defined in $(\ref{mousekokhay})$. For all $N\in\N$, we have
\label{May 15, 2017}
$$\varOmega(N)\leq\rad(N)^{2\operatorname{d}(N)}.$$
\end{lem}
\begin{proof}
From the proof of Theorem~\ref{b1}.($b$), 
we 
see that $\varOmega(N)<\Omega'(N)+\operatorname{d}(N)$ for all
$N>1$. 
By induction on the number of prime divisors of $N$, it follows easily that $\varOmega'(N)+ 
N^2\leq\rad(N)^{2\operatorname{d}(N)}$
for all $N>1$. 
\end{proof}


\section{The common multiple with the least weight} 
\label{27.08.2015}
In the previous section, we saw that 
if a holomorphic eta quotient on $\GN$
is not factorizable on $\GN$, 
then its weight is at most equal to $\kappa(N)/2$. In this section, we show that 
$\kappa(N)/2$ is 
the 
smallest possible weight
for an eta quotient $f$ such that
for each 
holomorphic eta quotient $g$ which is 
not factorizable on $\Gm_0(N)$,
$f/g$ is holomorphic (see Theorem~\ref{17.10Sept}). 

\begin{lem}
For $N\in\N$ and $t\in\D$, the holomorphic eta quotient $\e^{B_N(\un\hs,t)}$ is not factorizable on $\GN$,
where $B_N\in\Z^{\D\times\D}$ is as defined in $(\ref{28.08.2015})$.
\label{waybackhome}
\end{lem}

\begin{proof}
For $t\in\D$ and for 
$Y=A_N\cdot B_N(\un\hs,t)\in\ZD$, from (\ref{28.08.2015})
we get
$$
Y_d=\left\{\begin{array}{cl}
                       m_{{t,N}}&\text{if $d=t$}\\
0&\text{otherwise}
                      \end{array}\right.
$$
for all $d\in\D$. Recall that for $X\in\ZD$, $\eta^X$ 
is holomorphic if and only if $A_NX\geq 0$  (see \ref{28.04}).
Suppose, $\e^{B_N(\un\hs,t)}$ is factorizable on $\GN$. Then there are $X',X''\in\bb{Z}^{\D}\smallsetminus\{0\}$ 
with $B_N(\un\hs,t)=X'+X''$ such that 
$A_NX'\geq0$ and $A_NX''\geq0$. Hence, there exist $m',\hs m''>0$ 
with $m_{{t,N}}=m'+m''$  such that for $d\in\D$, we have
$$
(A_NX')_d=\left\{\begin{array}{cl}
                       m'&\text{if $d=t$,}\\
0&\text{otherwise}
                      \end{array}\right.
\hspace{.5cm}\text{and}\hspace{.6cm}
(A_NX'')_d=\left\{\begin{array}{cl}
                      m''&\text{if $d=t$,}\\
0&\text{otherwise.}
                      \end{array}\right.
$$
In other words, we have $X'= 
m'\cdot A_N^{-1}(\un\hs,d)$\hspace{2pt} and\hspace{2pt} $X'' 
={m''}\cdot A_N^{-1}(\un\hs,d)$. Since
${m'}, {m''}<{m_{{t,N}}}$ and since $m_{{t,N}}$ is the smallest positive integer such that 
$m_{{t,N}}\cdot A_N^{-1}(\un\hs,t)\in\bb{Z}^{\D}$, we conclude that $X'\notin\bb{Z}^{\D}$ and $X''\notin\bb{Z}^{\D}$. 
Thus, we get a contradiction! Hence,  $\eta^{B_N(\un\hs,t)}$ is not factorizable on $\GN$.
\end{proof}

\ \newline
\textit{Proof of Theorem \ref{17.10Sept}.} 
Let $F_N$ be the same as 
in (\ref{08.09.2015}). Then Lemma~\ref{08.09.2015A} and Lemma~\ref{b2}
together imply that if a holomorphic eta quotient $h$ on $\GN$ is divisible by $F_N$,
then it is divisible by all the holomorphic eta quotients on $\GN$ which are not factorizable on $\GN$.

Conversely, let a holomorphic eta quotient $h$ on $\Gamma_0(N)$ 
be divisible by each 
holomorphic eta quotient $g$ on $\GN$ which is not factorizable on $\Gamma_0(N)$.  
Let $B_N\in\Z^{\D\times\D}$ be as defined in $(\ref{28.08.2015})$.
Then Lemma~\ref{waybackhome} implies that 
$h$
is divisible by $\e^{B_N(\un\hs,t)}$  for all $t\in\D$.
So in particular, we have 
\begin{equation}
\operatorname{ord}_{1/t}(h\hs;\GN)\geq \operatorname{ord}_{1/t}(\e^{B_N(\un\hs,t)};\GN)=\operatorname{ord}_{1/t}(F_N;\GN),
\label{09.09.2015}\end{equation}
where the last equality 
holds since $F_N$ is the product of all the eta quotients 
$\{\e^{B_N(\un\hs,t)}\}_{\substack{\ \\ t\in\D}}$, 
and
since $(\ref{28.08.2015})$ and $(\ref{28.04})$ together imply that $\e^{B_N(\un\hs,t)}$ has nonzero order only at the cusp
$1/t$ of $\GN$.
Since any eta quotient on $\GN$ is uniquely determined by its orders at the set of the cusps
$\{1/t\}_{\substack{\ \\ t\in\D}}$ of $\GN$, from (\ref{09.09.2015})\hspace{1.3pt} 
it follows that $h$ is divisible by $F_N$.
\qed

\section{Examples of irreducible holomorphic eta quotients}
In this section, we shall show that there exist holomorphic eta quotients of arbitrarily large weights (see Theorem~\ref{10.09.2015}).

\begin{lem}
For $N\in\N$ and $t\in\cl{D}_{N/\rad(N)}$, the holomorphic eta quotient
$\e^{B_N(\un\hs,t)}$ is irreducible, where $B_N\in\Z^{\D\times\D}$ is as defined in $(\ref{28.08.2015})$.
\label{10.09.2015.1}\end{lem}
\begin{proof} 
From  Theorem~3 
in \cite{B-three}, we know that a holomorphic 
eta quotient on $\GN$ is reducible only if it is factorizable on some $\Gm_0(M)$ 
for some multiple $M$ of $N$ 
with $\rad(M)=\rad(N)$.
Suppose, for some $t\in\cl{D}_{N/\rad(N)}$ 
the \he $\eta^{B_N(\un\hs,t)}$
is reducible. Then there exists a multiple $M$ of $N$ with $\rad(M)=\rad(N)$
such that $\eta^{B_N(\un\hs,t)}$ is factorizable on $\Gm_0(M)$. Since 
$t\in\cl{D}_{N/\rad(N)}\subseteq\cl{D}_{M/\rad(M)}$, it follows from 
(\ref{31.08.2015}) and (\ref{equ1}) that $B_M(d,t)=B_N(d,t)$
for all $d\idi N$ and $B_M(d,t)=0$ if $d\nidi N$. In other words, we have
$\eta^{B_N(\un\hs,t)}=\e^{B_M(\un\hs,t)}$ which is not factorizable on $\Gm_0(M)$ by Lemma~\ref{waybackhome}.
Thus, we get a contradiction! Hence,  for all $t\in\cl{D}_{N/\rad(N)}$, 
 $\eta^{B_N(\un\hs,t)}$
is irreducible. 
\end{proof}
\ \newline
\textit{Proof of Theorem~\ref{10.09.2015}.} Since for all 
$X\in\ZD$, the weight of the eta quotient $\e^X$ is $\sigma(X)/2$, the theorem 
follows immediately from Lemma~\ref{10.09.2015.1}, (\ref{30.08.2015})
and from the fact that for $t=N/\rad(N)$, the eta quotient $\e^{B_N(\un\hs,t)}$ is of level~$N$ (see \ref{31.08.2015} and \ref{equ1}).
\qed

 \section*{Appendix: Comparison of the weights}
\label{26.08.2015}\hypertarget{mxwt}{}
By ${k_{\max}}(N)/2$, we denote 
the maximum of the weights of  \hes of level $N$ which are not factorizable on $\GN$.
Let $p$ be a prime. From 
the \hyperlink{levp}{discussion} about holomorphic eta quotients on $\Gm_0(p)$ in Section~\ref{intro},
it follows that 
${k_{\max}}(p)=p-1$. 
Also, from Theorem~6.4 in \cite{B-two}, we 
know ${k_{\max}}(p^2)=(p-1)^2$.
With the support of a huge amount of experimental data, we make the following conjecture:
\begin{conj}
$(a)$ For each prime number $p$, all the irreducible holomorphic eta quotients of level $p^3$
are rescalings of eta quotients of smaller levels. In particular,
we have 
${k_{\max}}(p^3)=(p-1)^2$.
\newline$(b)$ For each odd prime $p$ 
and for all integers $n>3$, we have
\begin{equation}
k_{\max}(p^n)=(n-1)(p-1)^2-2^{r_n}\Big(\Big{\lfloor}\frac{n}{2}\Big{\rfloor}(p-1)-1\Big),
\label{May 20, 2017}\end{equation}
where $r_n\in\{0,1\}$ is the residue of $n$ modulo $2$.
\end{conj}
For all odd primes $p$ and for 
all integers $n>3$,  in \cite{B-seven}\hspace*{1.1pt} we see examples of irreducible holomorphic eta quotients of level $p^n$ and
of the same weight as in $(\ref{May 20, 2017})$ 
(see Corollary~1 and (2.1) in \cite{B-seven}). However, 
the catch of the above problem 
is to show that any
holomorphic eta quotient of level $p^n$ whose
weight is greater than the 
quantity given in $(\ref{May 20, 2017})$,
must 
be reducible (see~Conjecture~1 in~\cite{B-seven} and Theorem~2 in~\cite{B-three}).

In the table below, we compare ${k_{\max}}(N)$ with $\kappa(N)$ for several 
\mbox{$N\in\N$},
where 
$\kappa(N)/2$ 
is the weight of the eta quotient~$F_N$\hs  
which we defined in Theorem~\ref{17.10Sept}\hspace*{1.2pt} (see also \ref{08.09.2015}). 
Since we have already discussed above the
cases of odd prime powers as well as those of $2^n$ for 
$n\leq3$, 
we 
omit such levels from the following table.
\begin{table}[h]
 \caption{$k_{\max}(N)$ vs. $\kappa(N)$}
\begin{center}
 \begin{tabular}{c|c|c} 
N&$k_{\max}$& $\kappa$ 
\\\hline
$2\cdot3$&2&8\\$2\cdot5$&4&16\\$2^2\cdot3$&3&12\\$2\cdot7$&6&24\\$3\cdot5$&8&32\\$2^4$&2&5\\$2\cdot3^2$&5&16\\$2^2\cdot5$&5&24\\$3\cdot7$&12&48\\$2^3\cdot3$&5&16\\$2\cdot13$&12&48\\
$2^2\cdot7$&8&36\\$2\cdot3\cdot5$&15&64\\$2^5$&2&6\\$2\cdot17$&16&64\\$2^2\cdot3^2$&6&24\\$2\cdot19$&18&72
\end{tabular}
\quad
  \begin{tabular}{c|c|c} 
N&$k_{\max}$& $\kappa$ 
\\\hline
$3\cdot13$&24&96\\$2^3\cdot5$&8&32\\
$2\cdot3\cdot7$&23&96\\$2^2\cdot11$&13&60\\$3^2\cdot5$&18&64\\$2\cdot23$&22&88\\$2^4\cdot3$&6&20\\$2\cdot5^2$&17&48\\$3\cdot17$&32&128\\$2^2\cdot13$&16&72\\$2\cdot3^3$&7&24\\$2^3\cdot7$&12&48\\$3\cdot19$&36&144\\$2^6$&3&7
\\$2\cdot3\cdot11$&38&160\\$2^2\cdot17$&20&96\\$2\cdot5\cdot7$&33&192
\end{tabular}
\quad
  \begin{tabular}{c|c|c} 
N&$k_{\max}$& $\kappa$ 
\\\hline
$2\cdot37$&36&144\\$2\cdot3\cdot13$&45&192\\$2^4\cdot5$&11&40
\\$5\cdot17$&64&256\\
$2^3\cdot11$&20&80\\$2\cdot47$&46&184
\\$2^5\cdot3$&8&24\\$2\cdot7^2$&37&96\\$3^2\cdot11$&45&160\\$2^2\cdot5^2$&25&72\\$2\cdot3\cdot17$&60&256\\$2^3\cdot13$&24&96
\\$3\cdot5\cdot7$&56&384\\$3\cdot37$&72&288\\$2^4\cdot7$&18&60
\\$2^7$&3&8
\\$7\cdot19$&108&432
\end{tabular}
\end{center}
\label{table.II}
\end{table}

\begin{center}
 \begin{tabular}{c|c|c} 
N&$k_{\max}$& $\kappa$ 
\\\hline
$3^3\cdot5$&32&96\\$2^3\cdot17$&34&128\\$2^2\cdot37$&48&216
\\$2\cdot3^4$&13&32\\$2\cdot5\cdot17$&85&512
\\$2^2\cdot43$&56&252
\end{tabular}
\quad
  \begin{tabular}{c|c|c} 
N&$k_{\max}$& $\kappa$ 
\\\hline
$2^4\cdot11$&30&100
\\$11\cdot19$&180&720\\$2^8$&4&9\\$2^9$&5&10\\
$2\cdot503$&502&2008\\$2^{10}$&6&11
\end{tabular}
\quad
  \begin{tabular}{c|c|c} 
N&$k_{\max}$& $\kappa$ 
\\\hline
$17\cdot97$&1536&6144
\\$2^{11}$&6&12\\
$2^{12}$&7&13\\$2^{13}$&7&14\\$2^{14}$&9&15\\$2^{15}$&9&16
\end{tabular}
\end{center}

\section*{Acknowledgments}
I am thankful to Sander Zwegers,
who asked during my talk at Cologne 
whether a Mersmann type finiteness theorem 
holds if we keep the level of the eta quotients fixed instead of their weight.
Corollary~\ref{b11} is precisely an answer to 
his question. 
I would like to thank 
Don Zagier,  Christian Wei\ss, Danylo Radchenko, Armin Straub,
Nadim Rustom and Christian Kaiser
for their comments.
I 
made
the computations for the above tables
using 
$\mathtt{PARI/GP}$~\cite{PARI2} and $\mathtt{Normaliz}$~\cite{Norm}
which I learnt to use from Don and
Danylo. I~am grateful to them
for acquainting me with 
these very useful computational tools.
In~particular, Danylo 
computed 
$\hyperlink{mxwt}{k_{\max}}(24)$, $\hyperlink{mxwt}{k_{\max}}(28)$ and $\hyperlink{mxwt}{k_{\max}}(30)$ 
for the above table.
I~am grateful 
to the CIRM~:~FBK (International Center for Mathematical Research of the Bruno Kessler Foundation) in Trento
for providing me 
with 
an office space and supporting me with a fellowship 
 during the preparation of this article.

\bibliography{fixedlevel-bibtex}
\bibliographystyle{IEEEtranS}
\nocite{*}

 \end{document}